\documentclass[11pt]{article}

\usepackage{hyperref}
\hypersetup{
	colorlinks,
	allcolors=blue 
}
\usepackage{amsmath}
\DeclareMathOperator*{\argmin}{arg\,min}
\usepackage{amssymb}
\usepackage{float}
\usepackage{epsfig}
\usepackage{graphicx}
\graphicspath{{./Pictures/}} 
\usepackage{epstopdf}
\usepackage{multirow}
\usepackage{subfigure}
\usepackage{threeparttable,booktabs}
\usepackage[utf8]{inputenc}      
\usepackage{enumitem}
\usepackage{authblk}

\usepackage{color}

\usepackage{amsthm}				
\newtheorem{theorem}{Theorem} 

\newtheorem{lemma}[theorem]{Lemma}

\theoremstyle{definition}

\theoremstyle{remark}
\newtheorem{remark}{Remark}
\newtheorem{example}[remark]{Example}

\newcommand{\abs}[1]{\left\vert {#1} \right\vert}
\newcommand{\Norm}[1]{\left\Vert {#1} \right\Vert}
\newcommand{\bull}{\raisebox{0.5\height}{{\tiny$\bullet$}}} 
\newcommand{\IR}{\mathbb{R}} 
\newcommand{\IN}{\mathbb{N}} 
\newcommand{\IZ}{\mathbb{Z}} 
\newcommand{\IC}{\mathbb{C}} 
\newcommand{\cont}[1]{\mathcal{C}^{#1}} 
\newcommand{\Rs}{\textsf{R}} 
\newcommand{\rec}[1]{\frac{1}{#1}} 

\begin{document}

\title{Eigensolutions and spectral analysis of a model for vertical gene transfer of plasmids}  

\author[a]{Eva Stadler}
\affil[a]{\small\textit{ Department of Mathematics, Technische Universität München, 85748 Garching, Germany}}

\date{}

\maketitle

\begin{abstract} 
	\noindent Plasmids are autonomously replicating genetic elements in bacteria. At cell division plasmids are distributed among the two daughter cells. This gene transfer from one generation to the next is called vertical gene transfer.
	We study the dynamics of a bacterial population carrying plasmids and are in particular interested in the long-time distribution of plasmids.
	Starting with a model for a bacterial population structured by the discrete number of plasmids, we proceed to the continuum limit in order to derive a continuous model.
	The model incorporates plasmid reproduction, division and death of bacteria, and distribution of plasmids at cell division. It is a hyperbolic integro-differential equation and a so-called growth-fragmentation-death model.
	As we are interested in the long-time distribution of plasmids we study the associated eigenproblem and show existence of eigensolutions.
	The stability of this solution is studied by analyzing the spectrum of the integro-differential operator given by the eigenproblem. 
	By relating the spectrum with the spectrum of an integral operator we find a simple real dominating eigenvalue with a non-negative corresponding eigenfunction.
	Moreover, we describe an iterative method for the numerical construction of the eigenfunction.	
\end{abstract}

\vspace{0.3cm}

\vspace{0.3cm}
\noindent\textbf{Mathematics Subject Classifier (2010)} 92D25 $\cdot$ 35L02 $\cdot$ 35Q92 $\cdot$ 47A10 $\cdot$ 47B65

\section{Introduction}
\label{S-Intro}

Plasmids are extrachromosomal genetic elements in bacteria which can replicate autonomously \cite{Casali2003}.
Plasmids have been intensively studied as they are used in biotechnology as vectors to amplify DNA and for recombinant protein production \cite{Clark2016}.
Moreover, the spread of plasmids is one mechanism by which antibiotic resistance genes spread in a bacterial population \cite{Casali2003,Beebee2008}.
For their biotechnological use it is important that bacteria do not loose plasmids as they cannot produce the desired proteins without the corresponding genetic information. It is also important that plasmids do not accumulate in bacteria as with a high plasmid load the metabolic burden is very high and the bacteria become inactive i.e. non-producing \cite{Bentley1990}.
Plasmid loss can be dealt with by adding antibiotic resistance genes to the plasmid genome and antibiotics to the growth medium of the bacteria such that only bacteria inheriting the plasmid can survive \cite{Clark2016}.
However, it is not possible to counteract plasmid accumulation in the same way.
Therefore, it is of great interest to study the distribution of plasmids in a bacterial population and to gain a deeper understanding of the mechanisms leading to plasmid accumulation.\par

Plasmids can be classified by their copy number i.e. the number of plasmids a cell contains. There are high-copy plasmids with hundreds of copies per cell and low-copy plasmids with only few copies per cell \cite{Casali2003}. 
It is usually assumed that high-copy plasmids are randomly distributed to the daughter cells at cell division.
However, there is also evidence that this is not the case and there are also segregation mechanisms for high-copy plasmids such as clustering of plasmids and localization at the cell poles \cite{Pogliano2001, MillionWeaver2014}. \par

There are many different ways to model the dynamics of plasmids, for example one can consider the bacterial population structured by the number of plasmids \cite{Ganusov2000, Mueller2017} or distinguish between the plasmid-free and the plasmid-bearing subpopulation \cite{Stewart1977}. 
There are discrete \cite{Bentley1993}, stochastic \cite{Mueller1982}, and continuous models \cite{Doumic2007, Calvez2012}. Continuous models for plasmid dynamics can be classified as aggregation-fragmentation or growth-fragmentation equations which have been studied extensively \cite{Calvez2012,Mischler2016}. 
There is also extensive literature on models of structured populations \cite{Magal2008,Calsina1995,Metz1986,Cushing1998} and of structured cell population dynamics \cite{Perthame2007,Doumic2007}. Arino \cite{Arino1995} gives an overview of different models and mathematical methods used for their analysis.
The mathematical methods used for the analysis of growth-fragmentation equations include theory of semigroups \cite{Webb2008}, the Laplace transform \cite{Heijmans1986}, and theory of positive operators \cite{Heijmans1986, Doumic2007}. \par

The aim of this paper is to extend results about existence of eigensolutions for a model of vertical gene transfer of high-copy plasmids. 
Furthermore, we study the stability of the eigensolution by analyzing the spectrum of the differential operator given by the model equation.
We focus on high-copy plasmids and derive a continuous model for vertical gene transfer of plasmids in a bacterial population.
As plasmids often contain genes that are necessary or beneficial for survival of the host-bacterium \cite{Summers1996} it makes sense that the death rate of bacteria without plasmids is higher than the death rate of bacteria containing plasmids. Therefore, we include a non-constant death rate depending on the number of plasmids into the model.
Bacteria with a high plasmid load also have a high metabolic burden \cite{Bentley1990} and therefore plasmids cannot reproduce to arbitrarily large numbers. This behavior can be modeled e.g. by a logistic plasmid reproduction rate.
The model we consider is given by a hyperbolic partial integro-differential equation and is similar to the model considered by \cite{Campillo2016, Doumic2007, Mueller2017, Calvez2012}. However, to the best of our knowledge there are no existence or stability results for this model with non-constant cell death rate and logistic plasmid reproduction rate.\par 

The paper is structured as follows: in Section \ref{S-Model} we develop a partial integro-differential equation model for a bacterial population structured by the number of plasmids. In Section \ref{S-Existence} we show existence of a solution for the associated eigenproblem. The stability of the solution of the eigenproblem is considered in Section \ref{S-Spectrum} where we show that there is a real dominant eigenvalue. In Section \ref{S-Numerics} we describe a method to construct and numerically simulate the eigensolution.

\section{Model}
\label{S-Model}

\subsection{Discrete model}
\label{S-DiscrMod}

We consider a bacterial population structured by the number of plasmids.
First, we formulate a discrete model and then in section \ref{S-ContMod} we proceed to a continuous model.

It turns out that plasmid segregation is central for the properties of the model, particularly if a cell contains only few plasmids at cell division.
We assume that there is a number $n$ such that during cell division bacteria with fewer than $n$ plasmids pass all their plasmids to one daughter cell and none to the other daughter cell. If a bacterium has only one plasmid at the time of cell division then it can give this plasmid only to one daughter and the other daughter receives no plasmids. Thus, with $n=2$ this assumption is reasonable.

The number of plasmids is denoted by $i\in\IN_0$. The bacterial population consists of bacteria with less than $n$ plasmids, $v_i(t)$ for $i\in\IN_0$ with $i<n$, and bacteria with at least $n$ plasmids, $w_i(t)$ for $i\in\IN$ with $i\geq n$. We denote the plasmid reproduction rate by $\tilde{b}(i)$, cell division rate by $\beta(i)$, and cell death rate by $\mu(i)$. At cell division a bacterium with at least $n$ plasmids distributes its plasmids to the two daughters. The daughter cells can be distinguished from one another as one daughter inherits the mother's younger pole (the poles are the respective ends of rod-shaped bacteria, they can be distinguished e.g. by their age \cite{Aakre2012}). 
If the mother contains $j\geq n$ plasmids, we denote by $p(i,j)$ the probability that the daughter with the mother's younger pole inherits $i$ plasmids.

We assume that $\tilde{b}(i)=0$ for $i=0$. For $j\geq n$ it holds that
\begin{align*}
\sum\limits_{i=0}^{j} p(i,j)+p(j-i,j) = 2.
\end{align*}
The model equations are then given by:
\begin{align}
	\dot{v}_0 &= (\beta(0)-\mu(0)) v_0 + \sum\limits_{j=1}^{n-1} \beta(j) v_{j}  + \sum\limits_{j=n}^{\infty} \beta(j)[p(0,j)+p(j,j)]w_j\label{DiscEqv0} \\
	\dot{v}_i &= -\mu(i) v_i + \tilde{b}(i-1) v_{i-1} - \tilde{b}(i) v_i +  \sum\limits_{j=n}^{\infty} \beta(j) [p(i,j)+p(j-i,j)]w_j\label{DiscEqv}\\
	\begin{split}\label{DiscEqu}
		\dot{w}_i &= -(\beta(i)+\mu(i)) w_i + \tilde{b}(i-1)w_{i-1} - \tilde{b}(i)w_i \\
		&\phantom{{}={}} + \sum\limits_{j=i}^{\infty} \beta(j) [p(i,j)+p(j-i,j)] w_j,
	\end{split}
\end{align}
where $w_{n-1}:= v_{n-1}$.

\subsection{Continuous model}
\label{S-ContMod}

Next, we want to proceed to the continuum limit, so we approximate $v_i(t)$ and $w_i(t)$ by smooth functions $v(z,t)$ resp. $w(z,t)$, i.e. for $h>0$ small
\begin{align*}
v_i(t) \approx \int\limits_{ih -\frac{h}{2}}^{ih+\frac{h}{2}} v(z,t)\: dz \approx v(ih,t)h \:\text{ resp. }\: w_i(t) \approx \int\limits_{ih -\frac{h}{2}}^{ih+\frac{h}{2}} w(z,t)\: dz \approx w(ih,t)h.
\end{align*}
Likewise, we approximate
\begin{itemize}
	\setlength\itemsep{0.5em}
	\item[\bull] $\dot{v}_i(t) \approx \partial_t v(ih,t) h$ and $\dot{w}_i(t) \approx \partial_t w(ih,t) h$ for $i\in\IN$,
	\item[\bull] $\tilde{b}(i) \approx b(ih)\frac{1}{h}$ for $i\in\IN_0$,
	\item[\bull] $\beta(i) = \beta(ih)$ and $\mu(i) = \mu(ih)$ for $i\in\IN_0$,
	\item[\bull] $p(i,j)+p(j-i,j) = k(ih,jh) h$ for $i\in\IN$, $j\in\IN$, $i<j$, $j\geq n$, and
	\item[\bull] $p(0,j)+p(j,j) = k_0(jh)$ for $j\in\IN$, $j\geq n$.
\end{itemize}
With these approximations, equation \eqref{DiscEqv0} becomes
\begin{align*}
\dot{v}_0 = (\beta(0)-\mu(0)) v_0 + \sum\limits_{j=1}^{n-1} \beta(jh) v(jh,t)h + \sum\limits_{j=n}^{\infty} \beta(jh) k_0(jh) \:w(jh,t)h.
\end{align*}
By defining $jh =: z'$ and $nh =: m>0$ and taking the limit $h\rightarrow 0$ we obtain
\begin{align*}
\dot{v}_0 = (\beta(0) -\mu(0)) v_0 + \int\limits_{0}^{m} \beta(z') v(z',t)\: dz' + \int\limits_{m}^{\infty} \beta(z') k_0(z') w(z',t)\: dz'.
\end{align*}
Analogously, we obtain from equations \eqref{DiscEqv} and \eqref{DiscEqu} for $z\in(0,m)$
\begin{align*}
\partial_t v(z,t) &+ \partial_z (b(z)v(z,t)) = -\mu(z) v(z,t) + \int\limits_{m}^{\infty} \beta(z') \: k(z,z')\: w(z',t)\: dz'
\end{align*}
and for $z\geq m$
\begin{align*}
&\partial_t w(z,t) + \partial_z (b(z)w(z,t)) =\\
&\quad = -(\beta(z)+\mu(z)) w(z,t) + \beta(z) k_0(z) w(z,t) + \int\limits_{z}^{\infty} \beta(z') \: k(z,z')\: w(z',t)\: dz'\\
&\quad =  -\left(\beta(z)(1-k_0(z))+\mu(z)\right) w(z,t) +  \int\limits_{z}^{\infty} \beta(z') \: k(z,z')\: w(z',t)\: dz'.
\end{align*}
We assume that bacteria with at least $m$ plasmids always pass at least one plasmid to each daughter at cell division, therefore $k_0(z)=0$ for all $z\geq m$. The case $k_0 \neq 0$ is analogous to $k_0(z)=0$.

At $z=0$ we add a zero flux boundary condition i.e. $\lim\limits_{z\rightarrow 0^+} b(z)v(z,t)=0$ for all $t\geq 0$.
The boundary condition at $z=m$ is given by $b(m)w(m,t)=\lim\limits_{z\rightarrow m^-} b(z)v(z,t)$ and therefore for continuous $b$ with $b(m)\neq 0$, $w(m,t)=\lim\limits_{z\rightarrow m^-} v(z,t)$ for all $t\geq 0$.

Overall, the continuous model reads
\begin{flalign}
\left\lbrace
\begin{aligned}
	&\dot{v}_0 (t) = (\beta(0) -\mu(0)) v_0 (t) + \int\limits_{0}^{m} \beta(z') v(z',t)\: dz',\\
	&\begin{aligned}
		&\begin{aligned}
			&\partial_t v(z,t) + \partial_z (b(z)v(z,t)) = -\mu(z) v(z,t) \\
			&\qquad + \int\limits_{m}^{\infty} \beta(z') k(z,z')w(z',t)\: dz',
		\end{aligned}\quad
		&\text{if } z\in(0,m),\\
		&\begin{aligned}
			&\partial_t w(z,t) + \partial_z (b(z)w(z,t)) = -(\beta(z)+\mu(z)) w(z,t)\\
			&\qquad + \int\limits_{z}^{\infty} \beta(z')k(z,z')w(z',t)\: dz',
		\end{aligned}\quad
		&\:\text{if } z\geq m,\\
	\end{aligned}\\
	& \lim\limits_{z\rightarrow 0^+} b(z)v(z,t)=0\quad \text{and} \quad w(m,t)=\lim\limits_{z\rightarrow m^{-}} v(z,t) \quad \text{for all }t\geq 0,\\
	& v(z,0) = \tilde{v}_0(z) \quad \text{and}\quad w(z,0)=w_0(z).
\end{aligned}
\right.\label{3Eq}
\end{flalign}

If we have a solution for the second and third equation of \eqref{3Eq} then we can find a solution for the first equation by variation of parameters. Therefore, in the following we only consider the latter equations.

Let $u(z,t) := \chi_{(0,m)}(z) v(t,z) + \chi_{[m,\infty)}(z) w(z,t)$ where $\chi$ denotes the characteristic function. Then, the second and third equation of \eqref{3Eq} can be combined to 
\begin{equation*}
	\begin{aligned}
		\partial_t u(z,t) + \partial_z(b(z)u(z,t)) &= - (\beta(z)\cdot\chi_{[m,\infty)}(z) + \mu(z)) u(z,t)\\
		&\phantom{{}={}} + \int\limits_{\max\{m,z\}}^{\infty} \beta(z') k(z,z') u(z',t)\: dz',
	\end{aligned}
\end{equation*}
for $z>0$. For the sake of brevity let $\beta_m(z) := \beta(z)\cdot \chi_{[m,z_0]}(z)$.\par

\medskip
We make the following assumptions on the parameters in model:
\begin{enumerate}[label=(A\arabic*),leftmargin=*]
	\item There is a $z_0\geq 1$ such that $b(0)=b(z_0)=0$, $b(z)>0$ for all $z\in (0,z_0)$, and $b\in \cont{1}([0,z_0])$. \label{Assb}
	\item $\beta \in \cont{0}([0,z_0])$ and $0<\underline{\beta}_m \leq \beta(z) \leq \overline{\beta}_m$ for all $z\in[m,z_0]$.\label{Assbeta}
	\item $\mu\in \cont{0}([0,z_0])$ and $0\leq \underline{\mu}\leq \mu(z)\leq \overline{\mu}$ for all $z\in[0,z_0]$.\label{Assmu}
	\item $k\in \cont{1}(\Omega)$, where $\Omega := \{ z,z'\in[0,z_0]:\: z\leq z'\text{ and } z'\geq m \}$, $k$ is symmetric in the sense that $k(z,z')=k(z'-z,z')$ for all $(z,z')\in \Omega$, $k(0,z')=k(z',z')=0$ for all $z'\in[m,z_0]$, and $\int_{0}^{z'} k(z,z')\: dz =2$ for all $z'\in[m,z_0]$.\label{Assk}
\end{enumerate}

\medskip
\noindent Assumptions \ref{Assb} - \ref{Assk} are assumed to hold throughout the remainder of this paper.

\medskip
\begin{example}\label{Ex-Ass}
	\begin{enumerate}[label=(\alph*),leftmargin=*]
		\item The conditions on the kernel $k$ are satisfied if there is a function $\Phi:[0,1]\rightarrow\IR$ with $\Phi\in \cont{1}([0,1])$, $\Phi(1-\xi)=\Phi(\xi)$ for all $\xi\in[0,1]$, $\Phi(0)=\Phi(1)=0$, and $\int_{0}^{1}\Phi(\xi)\:d\xi=1$ such that
		\begin{align*}
		k(z,z') = \frac{2}{z'} \Phi\left(\frac{z}{z'}\right)\chi_{\Omega}(z,z').
		\end{align*}
		We call a kernel of this type scalable (cf. \cite{Mueller2017}).
		A scalable kernel can be interpreted as bacteria (with at least $m$ plasmids) always distributing their plasmids in the same way independent of the number of plasmids they contain at cell division. The function $\Phi$ models how the plasmids are distributed to the daughters. For example, a bimodal distribution where bacteria give one daughter approximately 80\% of the plasmids and the other daughter the remaining plasmids is modeled by a function that is centered around 0.8 and 0.2.
		\item Constant cell division and death rates $\beta$ and $\mu$ with $\beta>0$ and $\mu\geq 0$ satisfy conditions \ref{Assbeta} and \ref{Assmu}. In general, we expect both the cell division and death rate to depend on the number of plasmids as e.g. it has been observed that the growth rate of bacteria decreases with increasing number of plasmids \cite{Bentley1990} and plasmids often carry genes necessary or beneficial for the survival of the bacterium such as resistance genes \cite{Summers1996}. 
	\end{enumerate}
\end{example}

\medskip
By assumption \ref{Assb}, if we start with a bacterial population where all bacteria contain at most $z_0$ plasmids, i.e. $w_0(z)=0$ respectively $u_0(z)=0$ for $z>z_0$, then the number of plasmids in a bacterium will never grow above $z_0$. Therefore, we consider for the remainder of this paper $z\leq z_0$.

We obtain the following model for a bacterial population structured by the number plasmids $z\in(0,z_0]$:
\begin{equation}
\left\lbrace\begin{aligned}
&\partial_t u(z,t) + \partial_z(b(z)u(z,t)) = -(\beta_m(z) + \mu(z)) u(z,t)\\
&\phantom{\partial_t u(z,t) + \partial_z(b(z)u(z,t)) = {}}+ \int\limits_{z}^{z_0} \beta_m(z') k(z,z') u(z',t)\: dz',\\
&\lim\limits_{z\rightarrow 0^+} b(z)u(z,t)=0 \text{ for all } t\geq 0, \quad u(z,0) = u_0(z).
\end{aligned}\right.\label{ModelEq}
\end{equation}

\medskip
\begin{remark}\label{R-AssPlasmids}
	Without the assumption that bacteria with few plasmids give all plasmids to one daughter at cell division we would have a singular integral kernel in \eqref{ModelEq}. For consistency, $\int_0^{z'} k(z,z')\: dz = 2$ for all $z'\in (0,z_0]$ and therefore $k(z,z')$ has a pole at $z'=0$. However, $\beta_m(z') k(z,z') = \beta(z')k(z,z')\chi_{[m,z_0]}(z')$ is bounded as it is equal to 0 for $z'<m$ and $\beta(z')$, $k(z,z')$ are assumed to be bounded for $z'\geq m$. Thus, we only require $\int_0^{z'} k(z,z')\: dz = 2$ for all $z'\in [m,z_0]$ and we do not have to deal with the technical difficulties that arise due to the singularity of $k$.
\end{remark}

\section{Existence of solutions to the eigenproblem}
\label{S-Existence}

It turns out that the long-term behavior is characterized by eigenfunctions. The eigenfunction $\mathcal{U}$ with eigenvalue $\lambda$ is a solution for the eigenproblem associated with \eqref{ModelEq}:
\begin{equation}
\left\lbrace\begin{aligned}
&(b(z)\mathcal{U}(z))' = -\left(\beta_m(z)+\mu(z)+\lambda\right)\mathcal{U}(z) + \int\limits_{z}^{z_0} \beta_m(z')k(z,z')\mathcal{U}(z') dz',\\
&\lim\limits_{z\rightarrow 0^+} b(z)\mathcal{U}(z)=0,\quad \mathcal{U}(z)\geq 0\text{ for all }z\in (0,z_0),\quad \int_{0}^{z_0} \mathcal{U}(z)\: dz =1.
\end{aligned}\right.\label{EigProb}
\end{equation}

\medskip
\noindent We define the flow $Z(t,z)$ by
\begin{equation}
\left\lbrace \begin{aligned}
&\frac{d}{dt} Z(t,z) = b(Z(t,z)),\quad\text{for all }t\geq 0, \: z\in[0,z_0],\\
&Z(0,z) = z,\quad\text{for all } z\in[0,z_0].
\end{aligned}\right.\label{Flow}
\end{equation}
The flow $Z(t,z)$ models the growth of plasmid in bacteria.
The Picard-Lindelöf theorem gives existence and uniqueness of the flow $Z(t,z)$ for $z\in[0,z_0]$ because $b\in\cont{1}([0,z_0])$ by \ref{Assb}.

\medskip
\begin{theorem}\label{ThmExistence}
	Under assumptions \ref{Assb}--\ref{Assk}, and the additional assumption
	\begin{enumerate}[label=(A\arabic*),leftmargin=*] 
		\setcounter{enumi}{4}
		\item $\int\limits_{0}^{\infty}\int\limits_{0}^{z_0} e^{-\int_{0}^{t} \underline{\mu}-\overline{\mu}+\beta_m\left( Z(s,z') \right)\:ds}\:dz'dt < \infty$, \label{Ass-ad}
	\end{enumerate}
	there exists a solution $(\lambda,\mathcal{U})$ for \eqref{EigProb} with $\lambda \in [-\overline{\mu},2\overline{\beta}_m + 2 -\underline{\beta}_m -\underline{\mu}]$ and $\mathcal{U}\in \cont{1}((0,z_0))$.
\end{theorem}

\medskip
\begin{remark}
	\begin{enumerate}[label=(\alph*),leftmargin=*]
		\item Note that if $\mu$ is constant, then Assumption \ref{Ass-ad} coincides with Assumption 3-4 of \cite{Campillo2016} and can therefore be interpreted in the same way: the flow $Z(t,z)$ has to move away from 0 or to be more exact out of the interval $[0,m]$ sufficiently fast (cf. \cite[Remark 7]{Campillo2016} and Example \ref{Ex-Ass-ad}\ref{Ex-Ass-ad-3}). See Example \ref{Ex-Ass-ad} for cases in which Assumption \ref{Ass-ad} holds.
		\item By change of variables $Z(t,z') \rightarrow z$,
		\begin{align*}
		\int\limits_{0}^{\infty}\int\limits_{0}^{z_0} e^{-\int_{0}^{t} \underline{\mu}-\overline{\mu}+\beta_m\left( Z(s,z') \right)\:ds}dz'dt = \int\limits_0^{z_0} \int\limits_{z'}^{z_0} e^{-\int_{z'}^{z}(\underline{\mu}-\overline{\mu}+\beta_m(y))/b(y)\: dy} \frac{dz}{b(z)}dz'.
		\end{align*}
		Therefore, Assumption \ref{Ass-ad} is equivalent to
		\begin{align*}
		\int\limits_0^{z_0} \int\limits_{z'}^{z_0} e^{-\int_{z'}^{z}(\underline{\mu}-\overline{\mu}+\beta_m(y))/b(y)\: dy} \rec{b(z)} \:dzdz' < \infty.
		\end{align*}
	\end{enumerate}
\end{remark}

\medskip
\begin{example}\label{Ex-Ass-ad}
	\begin{enumerate}[label=(\alph*),leftmargin=*]
		\item Let $\underline{\mu} - \overline{\mu} + \beta(y) \geq C >0$ for all $y\in[0,z_0]$ and $b(z)=\frac{b_0}{z_0} z(z_0-z)$ for some $b_0$, $z_0>0$. Then,
		\begin{align*}
		&\int\limits_0^{z_0}\int\limits_{z'}^{z_0} e^{-\int_{z'}^{z} (\underline{\mu}-\overline{\mu}+\beta_m(y))/b(y)\: dy}\: \rec{b(z)} \: dzdz' \leq\\
		&\qquad\leq \frac{z_0}{b_0}\int\limits_0^{z_0}\int\limits_{z'}^{z_0} e^{-C/b_0 \int_{z'}^{z} 1/(y(z_0-y))\:dy} \frac{1}{z(z_0-z)}\: dzdz'\\
		&\qquad = \frac{z_0}{b_0} \int\limits_0^{z_0}\int\limits_{z'}^{z_0} \frac{(z_0-z)^{C/b_0-1}}{z^{C/b_0+1}} \: dz\:  \left(\frac{z'}{z_0-z'}\right)^{C/b_0}\:dz' = \frac{z_0}{C} <\infty.
		\end{align*}
		Therefore, assumption \ref{Ass-ad} is satisfied e.g. if $\beta>0$, $\mu\geq 0$ are constant and $b(z)=\frac{b_0}{z_0} z(z_0-z)$ for some $b_0$, $z_0>0$ or if $\underline{\beta}_m > \overline{\mu} - \underline{\mu}$.
		\item If there is $0<M<z_0$ such that $\inf_{z\in[M,z_0]} \underline{\mu}-\overline{\mu}+\beta_m(z) >0$ and there exist $\varepsilon\in(0,1)$, $a>0$ s.t. $b(z)\geq a z^{1+\varepsilon}$ in a neighborhood of 0, then Assumption \ref{Ass-ad} holds (cf. \cite[Proposition 7]{Campillo2016}).\label{Ex-Ass-ad-3}
	\end{enumerate}
\end{example}

\subsection{Proof of Theorem \ref{ThmExistence}}
\label{S-Proof}

\noindent The proof of Theorem \ref{ThmExistence} follows \cite{Campillo2016, Doumic2007} and extends the results to the case of non-constant cell death rate $\mu$.

\medskip
\begin{proof}[Proof of Theorem \ref{ThmExistence}] 
	$\quad$\newline
	\indent \textit{Step 1: Regularization}\\
	We introduce the regularization parameter $\varepsilon\geq 0$ and define for $z\in[0,z_0]$ and $z'\in[m,z_0]$
	\begin{align*}
	B_{\varepsilon}(z) := \beta_m(z)+\mu(z)+\varepsilon \quad \text{and} \quad \kappa_{\varepsilon}(z,z'):=\beta_m(z')k(z,z')+\frac{2\varepsilon}{z_0}.
	\end{align*}
	We consider the Banach space $\cont{0}([0,z_0])$ with the supremum norm $\Norm{\cdot}_{\infty}$.
	Let
	\begin{align*}
	\underline{\lambda}_{\varepsilon} := \inf\left\lbrace \lambda\in\IR: \int\limits_{0}^{\infty}\int\limits_{0}^{z_0}e^{-\int_{0}^{t} \lambda + B_{\varepsilon}(Z(s,z'))\:ds}\:dz'dt <\infty \right\rbrace.
	\end{align*}
	Since $\lambda + B_{\varepsilon}(Z(s,z'))\leq \lambda + \overline{\beta}_m + \overline{\mu}+\varepsilon$ for all $z'\in [0,z_0]$, $s\in [0,\infty)$,
	\begin{align*}
	\int\limits_{0}^{\infty}\int\limits_{0}^{z_0} e^{-t(\lambda+\overline{\beta}_m+\overline{\mu}+\varepsilon)}\:dz'dt \leq \int\limits_{0}^{\infty}\int\limits_{0}^{z_0}e^{-\int_{0}^{t} \lambda + B_{\varepsilon}(Z(s,z'))\:ds}\:dz'dt.
	\end{align*}
	For $\lambda \leq -(\overline{\beta}_m+\overline{\mu}+\varepsilon)$ the integral on the left-hand side diverges, therefore we know that $\underline{\lambda}_{\varepsilon}\geq -(\overline{\beta}_m+\overline{\mu}+\varepsilon)$.
	Since $\lambda + B_{\varepsilon}(Z(s,z'))\geq \underline{\mu}-\overline{\mu}+\beta_m(Z(s,z'))$ for $\lambda >-\overline{\mu}$ and for all $z'\in[0,z_0]$, $s\in[0,\infty)$, by Assumption \ref{Ass-ad},
	\begin{align*}
	\int\limits_{0}^{\infty}\int\limits_{0}^{z_0}e^{-\int_{0}^{t} \lambda + B_{\varepsilon}(Z(s,z'))\:ds}\:dz'dt \leq \int\limits_{0}^{\infty}\int\limits_{0}^{z_0}e^{-\int_{0}^{t} \underline{\mu}-\overline{\mu}+\beta_m(Z(s,z'))\:ds}\:dz'dt<\infty.
	\end{align*}
	Therefore, $\underline{\lambda}_{\varepsilon}\in[-(\overline{\beta}_m+\overline{\mu}+\varepsilon),-\overline{\mu}]$.\\
	Let $\varepsilon\geq 0$ and $\lambda >\underline{\lambda}_{\varepsilon}$ for the remainder of this proof. For $f\in\cont{0}([0,z_0])$ define the operator $G_{\lambda}^{\varepsilon}$ by
	\begin{align*}
	G_{\lambda}^{\varepsilon}[f](z) = \int\limits_{0}^{\infty}\int\limits_{0}^{z_0} \kappa_{\varepsilon}(z,Z(t,z')) f(z') e^{-\int_{0}^{t} \lambda + B_{\varepsilon}(Z(s,z'))\: ds}\:dz'dt.
	\end{align*}
	$G_{\lambda}^{\varepsilon}$ is well-defined for $\lambda>\underline{\lambda}_{\varepsilon}$, therefore it is in particular well-defined for $\lambda >-\overline{\mu}$.
	
	\indent \textit{Step 2: Compactness of} $G_{\lambda}^{\varepsilon}$\\
	The operator $G_{\lambda}^{\varepsilon}$ maps continuous functions to continuous functions as for every $z'\in [0,z_0]$, $\beta_m(z')k(z,z')$ is continuous in $z$. Furthermore, for every $f\in\cont{0}([0,z_0])$,
	\begin{align*}
	\Norm{G_{\lambda}^{\varepsilon}[f]}_{\infty} \leq \left( \overline{\beta}_m\Norm{k}_{\infty} + \frac{2\varepsilon}{z_0} \right) \Norm{f}_{\infty} \int\limits_{0}^{\infty}\int\limits_{0}^{z_0}e^{-\int_{0}^{t} \lambda + B_{\varepsilon}(Z(s,z'))\:ds}\:dz'dt,
	\end{align*}
	where the last factor is finite because $\lambda >\underline{\lambda}_{\varepsilon}$.\\
	Let $f\in\cont{0}([0,z_0])$ and $z_1$, $z_2\in[0,z_0]$, then
	\begin{align*}
	\abs{G_{\lambda}^{\varepsilon}[f](z_1)-G_{\lambda}^{\varepsilon}[f](z_2)} &\leq \overline{\beta}_m \Norm{k}_{\cont{1}(\Omega)} \abs{z_1-z_2} \Norm{f}_{\infty}\\
	&\phantom{{}\leq{}} \cdot \int\limits_{0}^{\infty}\int\limits_{0}^{z_0}e^{-\int_{0}^{t} \lambda + B_{\varepsilon}(Z(s,z'))\:ds}\:dz'dt.
	\end{align*}
	By the Theorem of Arzelà-Ascoli, the operator $G_{\lambda}^{\varepsilon}: \cont{0}([0,z_0])\rightarrow \cont{0}([0,z_0])$ is therefore compact for every $\varepsilon \geq 0$ and $\lambda > \underline{\lambda}_{\varepsilon}$.
	
	\indent \textit{Step 3: Eigenelements}\\
	For every $\varepsilon>0$, the operator $G_{\lambda}^{\varepsilon}$ is a strictly positive operator in the sense that for every $f\in \cont{0}([0,z_0])$ with $f\geq 0$, $f\neq 0$, $G_{\lambda}^{\varepsilon}[f](z)>0$ for all $z\in[0,z_0]$. As $G_{\lambda}^{\varepsilon}$ is compact and strictly positive for $\varepsilon>0$ and $\lambda>\underline{\lambda}_{\varepsilon}$, by Krein-Rutman Theorem (see e.g. \cite[Theorem 1]{Dautray1990})
	the spectral radius $r(G_{\lambda}^{\varepsilon})>0$ is a simple eigenvalue of $G_{\lambda}^{\varepsilon}$ and there is a unique positive eigenfunction $\Psi_{\lambda}^{\varepsilon}\in \cont{0}([0,z_0])$ with $\Norm{\Psi_{\lambda}^{\varepsilon}}_{\infty}=1$.\\
	The map $\lambda \mapsto r(G_{\lambda}^{\varepsilon})$ is continuous (see e.g. \cite{Campillo2016}).
	
	\indent \textit{Step 4: Fixed point of $G_{\lambda}^{0}$}\\
	Let $f\in\cont{0}([0,z_0])$ with $f\geq 0$ and $\lambda >-\overline{\mu}$. We integrate $G_{\lambda}^{\varepsilon}[f]$ w.r.t. $z$,
	\begin{align*}
	\int\limits_{0}^{z_0} G_{\lambda}^{\varepsilon}[f](z)\:dz &= \int\limits_{0}^{\infty} \int\limits_{0}^{z_0} (2\beta_m(Z(t,z'))+2\varepsilon) f(z') e^{-\int_{0}^{t} \lambda + B_{\varepsilon}(Z(s,z'))\: ds}\: dz'dt\\
	&\leq (2\overline{\beta}_m+2\varepsilon) \int\limits_{0}^{\infty}\int\limits_{0}^{z_0} f(z') e^{-t(\lambda + \underline{\beta}_m +\underline{\mu})}\: dz'dt\\
	&\leq \frac{2\overline{\beta}_m+2\varepsilon}{\lambda + \underline{\beta}_m + \underline{\mu}} \int\limits_{0}^{z_0} f(z')\: dz'.
	\end{align*}
	Let $f=\Psi_{\lambda}^{\varepsilon}$,
	\begin{align*}
	r(G_{\lambda}^{\varepsilon}) \int\limits_{0}^{z_0} \Psi_{\lambda}^{\varepsilon}(z) \: dz \leq \frac{2\overline{\beta}_m+2\varepsilon}{\lambda+ \underline{\beta}_m  + \underline{\mu}} \int\limits_{0}^{z_0} \Psi_{\lambda}^{\varepsilon}(z') \: dz'.
	\end{align*}
	Therefore, $r(G_{\lambda}^{\varepsilon})\leq 1$ for $\lambda \geq 2\overline{\beta}_m + 2\varepsilon - \underline{\beta}_m  -\underline{\mu}$.\\
	In a similar way, we obtain a lower bound for $r(G_{\lambda}^{\varepsilon})$,
	\begin{align*}
	\int\limits_{0}^{z_0} G_{\lambda}^{\varepsilon}[f](z)\:dz &= \int\limits_{0}^{\infty} \int\limits_{0}^{z_0} (2\beta_m(Z(t,z'))+2\varepsilon) f(z') e^{-\int_{0}^{t} \lambda + B_{\varepsilon}(Z(s,z'))\: ds}\: dz'dt\\
	&\geq 2\int\limits_{0}^{\infty} \int\limits_{0}^{z_0} e^{-t(\lambda+\overline{\mu})} f(z') \partial_t\left( -e^{-\int_{0}^{t}\beta_m(Z(s,z'))+\varepsilon\: ds} \right)\: dz'dt.
	\end{align*}
	Recall that we assume $\lambda>-\overline{\mu}$. Integration by parts yields
	\begin{align*}
	\int\limits_{0}^{z_0} G_{\lambda}^{\varepsilon}[f](z)\:dz &\geq 2 \int\limits_{0}^{z_0} f(z')\: dz'\\
	&\phantom{{}\geq{}}- 2 \int\limits_{0}^{\infty}\int\limits_{0}^{z_0} f(z') (\lambda+\overline{\mu}) e^{-\int_{0}^{t}\lambda +\overline{\mu}+\beta_m(Z(s,z'))+\varepsilon\: ds}\: dz'dt.
	\end{align*}
	With $f=\Psi_{\lambda}^{\varepsilon}$ and integration over $z$,
	\begin{align*}
	&r(G_{\lambda}^{\varepsilon}) \int\limits_{0}^{z_0} \Psi_{\lambda}^{\varepsilon}(z)\: dz \geq \\
	&\quad \geq 2\int\limits_{0}^{z_0} \Psi_{\lambda}^{\varepsilon}(z')\:dz' - 2z_0 (\lambda + \overline{\mu}) \int\limits_{0}^{\infty}\int\limits_{0}^{z_0} \Psi_{\lambda}^{\varepsilon}(z') e^{-\int_{0}^{t} \lambda +\overline{\mu}+\beta_m(Z(s,z'))+\varepsilon\: ds}\:dz'dt\\
	&\quad \geq 2\int\limits_{0}^{z_0} \Psi_{\lambda}^{\varepsilon}(z')\:dz' - 2z_0 (\lambda + \overline{\mu}) \int\limits_{0}^{\infty}\int\limits_{0}^{z_0} \Psi_{\lambda}^{\varepsilon}(z') e^{-\int_{0}^{t} \lambda + B_{\varepsilon}(Z(s,z'))\: ds}\:dz'dt.
	\end{align*}
	Hence, $\lim\limits_{\lambda \rightarrow -\overline{\mu}^+} r(G_{\lambda}^{\varepsilon}) \geq 2$. Due to continuity of the map $\lambda \mapsto r(G_{\lambda}^{\varepsilon})$, there is a $\lambda_{\varepsilon}\in [-\overline{\mu},2\overline{\beta}_m + 2\varepsilon -\underline{\beta}_m -\underline{\mu}]$ such that $r(G_{\lambda_{\varepsilon}}^{\varepsilon}) = 1$ for every $\varepsilon>0$. Define $\Psi_{\lambda_{\varepsilon}}^{\varepsilon}=:\Psi_{\varepsilon}$.\\
	For every $\varepsilon\in(0,1]$ the operator $G_{\lambda_{\varepsilon}}^{\varepsilon}$ has the fixed point $\Psi_{\varepsilon}\in \cont{0}([0,z_0])$ with $\Psi_{\varepsilon} >0$, $\Norm{\Psi_{\varepsilon}}_{\infty}=1$, and $\lambda_{\varepsilon}\in [-\overline{\mu},2\overline{\beta}_m + 2 -\underline{\beta}_m -\underline{\mu}]$. The family $(\Psi_{\varepsilon})_{0<\varepsilon\leq 1}$ is uniformly bounded by 1. Let $z_1$, $z_2\in [0,z_0]$, we use $\Psi_{\varepsilon} = G_{\lambda_{\varepsilon}}^{\varepsilon}(\Psi_{\varepsilon})$ and the same estimates as in the second step to obtain
	\begin{align*}
	\abs{\Psi_{\varepsilon}(z_1)-\Psi_{\varepsilon}(z_2)} &\leq \overline{\beta}_m \Norm{k}_{\cont{1}(\Omega)} \abs{z_1-z_2} \int\limits_0^{\infty}\int\limits_0^{z_0} e^{-\int_{0}^{t} \lambda_{\varepsilon} + B_{\varepsilon}(Z(s,z'))\: ds}\: dz'dt\\
	&\leq \overline{\beta}_m \Norm{k}_{\cont{1}(\Omega)} \abs{z_1-z_2} \int\limits_0^{\infty}\int\limits_0^{z_0} e^{-\int_{0}^{t} \underline{\mu}-\overline{\mu} + \beta_m(Z(s,z'))\: ds}\: dz'dt,
	\end{align*}
	where the last factor is bounded by Assumption \ref{Ass-ad}.
	Thus, the family $(\Psi_{\varepsilon})_{0<\varepsilon\leq 1}$ is equicontinuous and compact by the Theorem of Arzelà-Ascoli.
	Hence, we can extract a subsequence of $(\lambda_{\varepsilon},\Psi_{\varepsilon})_{0<\varepsilon\leq 1}$ which converges for $\varepsilon\rightarrow 0$ to $(\lambda, \Psi)\in [-\overline{\mu},2\overline{\beta}_m + 2 -\underline{\beta}_m -\underline{\mu}] \times \cont{0}([0,z_0])$ with $\Psi\geq 0$ and $\Norm{\Psi}_{\infty}=1$.
	By Assumption \ref{Ass-ad} and dominated convergence,
	\begin{align*}
	\Psi(z) = \int\limits_{0}^{\infty}\int\limits_{0}^{z_0} \beta_m(Z(t,x))k(z,Z(t,x)) \Psi(x) e^{-\int_{0}^{t} \lambda + \beta_m(Z(s,x)) + \mu(Z(s,x))\: ds}\: dxdt.
	\end{align*}
	
	\indent \textit{Step 5: Conclusion}\\
	For $z\in(0,z_0)$ define
	\begin{align}
	\mathcal{U}(z) := \frac{1}{b(z)} \int\limits_{0}^{z}\Psi(y) e^{-\int_{y}^{z}\frac{\lambda +\beta_m(s) +\mu(s)}{b(s)}\:ds}\:dy.\label{DefU}
	\end{align}
	Then, $\mathcal{U}(z)\geq 0$ for all $z\in(0,z_0)$, $\mathcal{U}\in \cont{0}((0,z_0))$, and
	\begin{align*}
	\frac{d}{dz} \left(b(z)\mathcal{U}(z)\right) = -(\lambda+\beta_m(z)+\mu(z))\mathcal{U}(z) +\Psi(z).
	\end{align*}
	With change of variables $Z(t,z')\rightarrow y$ and $Z(s,z')\rightarrow w$,
	\begin{align*}
	\Psi(z) &= \int\limits_{0}^{\infty}\int\limits_{0}^{z_0} \beta_m(Z(t,x))k(z,Z(t,x))\Psi(x) e^{-\int_{0}^{t}\lambda+\beta_m(Z(s,x))+\mu(Z(s,x))\: ds}\:dxdt\\
	&=\int\limits_{0}^{z_0} \int\limits_{z'}^{z_0} \beta_m(y)k(z,y)\Psi(x) e^{-\int_{x}^{y} \frac{\lambda + \beta_m(w)+\mu(w)}{b(w)}\:dw} \frac{1}{b(y)}\: dydx\\
	&=\int\limits_{0}^{z_0}\beta_m(y)k(z,y) \frac{1}{b(y)}\int\limits_{0}^{y}\Psi(x)e^{-\int_{x}^{y} \frac{\lambda + \beta_m(w)+\mu(w)}{b(w)}\:dw}\:dxdy\\
	&= \int\limits_{z}^{z_0}\beta_m(y)k(z,y)\mathcal{U}(y) \:dy.
	\end{align*}
	Therefore, $\mathcal{U}$ defined by \eqref{DefU} is a solution for the PDE in \eqref{EigProb}.
	Using again change of variables, $\Norm{\Psi}_{\infty}=1$, $\lambda\in [-\overline{\mu},2\overline{\beta}_m + 2 -\underline{\beta}_m -\underline{\mu}]$, and Assumption \ref{Ass-ad},
	\begin{align*}
	\int\limits_{0}^{z_0} \mathcal{U}(z) \: dz &= \int\limits_{0}^{z_0} \frac{1}{b(z)} \int\limits_{0}^{z}\Psi(y) e^{-\int_{y}^{z}\frac{\lambda +\beta_m(s) +\mu(s)}{b(s)}\:ds}\:dydz\\
	&\leq \Norm{\Psi}_{\infty} \int\limits_{0}^{z_0}\int\limits_{y}^{z_0} e^{-\int_{y}^{z} (\lambda + \beta_m(s)+\mu(s))\:\frac{ds}{b(s)}}\:\frac{dz}{b(z)}dy\\
	&= \int\limits_{0}^{z_0}\int\limits_{0}^{\infty} e^{-\int_{0}^{t} \lambda + \beta_m(Z(w,y))+\mu(Z(w,y))\: dw}\: dtdy\\
	&\leq \int\limits_{0}^{z_0}\int\limits_{0}^{\infty} e^{-\int_{0}^{t} \underline{\mu}-\overline{\mu}+\beta_m(Z(w,y))\:dw}\: dtdy <\infty.
	\end{align*}
	Therefore, $\tilde{\mathcal{U}}(z)=\frac{\mathcal{U}(z)}{\int_{0}^{z_0}\mathcal{U}(z')\:dz'}$ is a solution of \eqref{EigProb}.
\end{proof}

\section{Spectral analysis}\label{S-Spectrum}

Consider the operator $A: \mathcal{W}\subseteq L^1((0,z_0)) \rightarrow L^1((0,z_0))$ defined by \begin{align*}
A[\mathcal{U}](z) := -(b(z)\mathcal{U}(z))' - (\beta_m(z)+\mu(z))\mathcal{U}(z) + \int\limits_z^{z_0} \beta_m(z') k(z,z') \mathcal{U}(z')\: dz',
\end{align*}
where $\mathcal{W} := \left\lbrace f\in L^1((0,z_0)): \: (b\cdot f)'\in L^1((0,z_0)) \right\rbrace$ with the norm
\begin{align*}
\Norm{f}_{\mathcal{W}} := \Norm{f}_{L^1((0,z_0))} + \Norm{(b\cdot f)'}_{L^1((0,z_0))}.
\end{align*}

Let $I: \mathcal{W}\rightarrow L^1((0,z_0))$, $f\mapsto f$ be the embedding of $\mathcal{W}$ into $L^1((0,z_0))$ and define $R_{\xi}: \mathcal{W} \rightarrow L^1((0,z_0))$ by $R_{\xi} := (\xi I - A)$. $R_{\xi}$ is invertible if and only if for every $f\in L^1((0,z_0))$ there is a unique $U\in\mathcal{W}$ such that
\begin{align*}
\frac{d}{dz} (b(z)U(z)) + (\xi + \beta_m(z)+\mu(z))U(z) - \int\limits_z^{z_0} \beta_m(z') k(z,z') U(z')\: dz' = f(z). 
\end{align*}
We use the transform $v(z):= b(z)U(z)$, then $v\in \mathcal{W}_v $ and 
\begin{align*}
v'(z) + \frac{\xi+\beta_m(z)+\mu(z)}{b(z)} v(z) - \int\limits_z^{z_0} \frac{\beta_m(z') k(z,z')}{b(z')} v(z')\: dz' = f(z),
\end{align*}
where $ \mathcal{W}_v := \left\lbrace f\in L^1\left((0,z_0),\frac{dz}{b(z)}\right):\: f'\in L^1((0,z_0)) \right\rbrace$.
For the sake of brevity we define
\begin{align*}
\alpha(z) :=  \frac{\xi+\beta_m(z)+\mu(z)}{b(z)} \quad \text{and} \quad L^1_w := L^1\left( (0,z_0), \frac{dz}{b(z)}\right).
\end{align*}
The weighted $L^1$-space $L^1_w$ is a Banach space with the norm 
\begin{align*}
\Norm{f}_{L^1_w} := \int\limits_0^{z_0} \abs{f(z)} \frac{dz}{b(z)}.
\end{align*}

Variation of parameters and $v(0)=0$ (as $\lim\limits_{z\rightarrow 0^+} b(z)\mathcal{U}(z)=0$ for a solution $\mathcal{U}$ for \eqref{EigProb}) yields
\begin{align}
v(z) = \int\limits_0^z \int\limits_x^{z_0} \frac{\beta_m(z')k(x,z')}{b(z')} v(z') \: dz' \: e^{-\int_{x}^{z} \alpha(y) \: dy}\: dx + \int\limits_0^z f(x) e^{-\int_{x}^{z} \alpha(y) \: dy}\: dx.\label{Eq-v}
\end{align}
For $v\in L^1_w$ we define the operator $T_{\xi}$ by 
\begin{align*}
T_{\xi}[v](z) := \int\limits_0^z \int\limits_x^{z_0} \frac{\beta_m(z')k(x,z')}{b(z')} v(z') \: dz' \: e^{-\int_{x}^{z} \alpha(y) \: dy}\: dx.
\end{align*}
Note that we have the index $\xi$ in $T_{\xi}$ as $\alpha$ depends on $\xi$. 

\medskip
\begin{lemma}\label{Lem-Tspace}
	For $\Re(\xi)>-(\underline{\beta}_m+\underline{\mu})$, $T_{\xi}: L^1_w \rightarrow L^1_w$ is bounded with operator norm $\Norm{T_{\xi}}\leq \frac{2\overline{\beta}_m}{\Re(\xi) + \underline{\beta}_m + \underline{\mu}}$.
\end{lemma}

\medskip
\begin{proof}
	Let $\alpha_r(z) := \Re(\alpha(z))$, $\xi_r:=\Re(\xi)$, and $v\in L^1_w$, then
	\begin{align*}
	\Norm{T_{\xi}[v]}_{L^1_w} &\leq \int\limits_0^{z_0} \rec{b(z)} \int\limits_0^{z} \int\limits_x^{z_0} \frac{\beta_m(z')k(x,z')}{b(z')} \abs{v(z')} \: dz' e^{-\int_x^z \alpha_r(y)\: dy} \:dxdz\\
	&= \int\limits_0^{z_0} \int\limits_x^{z_0} \int\limits_x^{z_0} \rec{b(z)} e^{-\int_x^z \alpha_r(y)\: dy}\: dz\: \frac{\beta_m(z')k(x,z')}{b(z')} \abs{v(z')} \:dz'dx.
	\end{align*}
	Since
	\begin{align*}
	\frac{d}{dz} \left( e^{-\int_x^z \alpha_r(y)\: dy} \right) &= -\alpha_r(z) e^{-\int_x^z \alpha_r(y)\: dy}\\
	&= -\frac{\xi_r + \beta_m(z)+ \mu(z)}{b(z)} e^{-\int_x^z \alpha_r(y)\: dy},
	\end{align*}
	we obtain for $\xi_r > -(\underline{\beta}_m+ \underline{\mu})$,
	\begin{align*}
	\rec{b(z)} e^{-\int_x^z \alpha_r(y)\: dy} \leq -\frac{d}{dz} \left( e^{-\int_x^z \alpha_r(y)\: dy} \right) (\xi_r + \underline{\beta}_m + \underline{\mu})^{-1}.
	\end{align*}
	Thus,
	\begin{align*}
	\Norm{T_{\xi}[v]}_{L^1_w} &\leq  (\xi_r + \underline{\beta}_m + \underline{\mu})^{-1} \int\limits_0^{z_0} \int\limits_x^{z_0} \frac{\beta_m(z')k(x,z')}{b(z')} \abs{v(z')} \:dz'dx\\
	&\leq \frac{\overline{\beta}_m}{\xi_r + \underline{\beta}_m + \underline{\mu}} \int\limits_0^{z_0} \int\limits_0^{z'} k(x,z')\:dx \: \frac{\abs{v(z')}}{b(z')}\: dz' \leq \frac{2\overline{\beta}_m}{\xi_r + \underline{\beta}_m + \underline{\mu}} \Norm{v}_{L^1_w}.
	\end{align*}
	Therefore, $\Norm{T_{\xi}[v]}_{L^1_w}<\infty$ for $\Re(\xi)>-(\underline{\beta}_m+\underline{\mu})$ and $T_{\xi}[v]\in L^1_w$. Furthermore, the operator norm of $T_{\xi}$ denoted by $\Norm{T_{\xi}}$ satisfies $\Norm{T_{\xi}}\leq \frac{2\overline{\beta}_m}{\xi_r + \underline{\beta}_m + \underline{\mu}}$. 
\end{proof}
	
	\bigskip
	In the following lemma we consider the connection of the spectrum of $T_{\xi}$ and $A$. We analyze the spectrum of $T_{\xi}$ and use this lemma to draw conclusions about the spectrum of $A$.

\medskip
\begin{lemma}\label{Lem-Spec}
	For every $\xi\in\IC$ it holds that $\xi$ is in the resolvent set of $A$, $\rho(A)$, if and only if $1\in \rho(T_{\xi})$ or equivalently $\xi\in\sigma(A)$ iff $1\in\sigma(T_{\xi})$ where $\sigma(A)$ denotes the spectrum of the operator $A$.
\end{lemma}

\medskip
\begin{proof}
	Let $\xi\in \IC$, then $\xi\in \rho(A)$ if and only if the operator $R_{\xi}: \mathcal{W} \rightarrow L^1((0,z_0))$ defined by $R_{\xi} := (\xi I - A) $ is invertible, its inverse exists and is an everywhere defined bounded linear operator \cite[p. 162]{Webb1985}.
	
	By the transformation $v(z) = b(z)u(z)$ we see that $R_{\xi}$ is invertible if and only if $(I-T_{\xi})$ is invertible (see equation \eqref{Eq-v}) which holds iff $1\in\rho(T_{\xi})$. 
\end{proof}

\medskip
\begin{lemma}\label{Lem-Res}
	If $\xi\in\IC$ with $\Re(\xi)>2\overline{\beta}_m - \underline{\beta}_m -\underline{\mu}$, then $\xi\in\rho(A)$.
\end{lemma}

\medskip
\begin{proof}				
	By Lemma \ref{Lem-Tspace},
	\begin{align*}
	\Norm{T_{\xi}} \leq \frac{2\overline{\beta}_m}{\xi_r + \underline{\beta}_m+\underline{\mu}},
	\end{align*}
	where $\xi_r = \Re(\xi)$.
	If $\Norm{T_{\xi}} < 1 $, then $1\in \rho(T_{\xi})$ and therefore $\xi \in\rho(A)$ by Lemma \ref{Lem-Spec}. For $\xi_r > 2\overline{\beta}_m - \underline{\beta}_m-\underline{\mu}$, $\Norm{T_{\xi}} < 1 $. 
\end{proof}

\medskip
\begin{lemma}\label{Lem-Tcomp}
	Let $\Re(\xi)=\xi_r > -(\underline{\beta}_m + \underline{\mu})$ and assume \ref{Ass-ad} holds, then the operator $T_{\xi}$ is compact.
\end{lemma}

\medskip
\begin{proof}
	We want to apply the Kolmogorov-Riesz-Fréchet Theorem to prove compactness of $T_{\xi}$ (see e.g. \cite[Theorem 4.26]{Brezis2010}).
	To this end, we extend $\beta_m$, $\mu$, $k$, $b$, and $\alpha$ to $\IR$ by setting them to 0 outside of the interval $[0,z_0]$ respectively for $k$ outside of $\Omega$.
	
	The operator $T_{\xi}$ is a bounded linear operator for $\xi_r > -(\underline{\beta}_m + \underline{\mu})$ (see Lemma \ref{Lem-Tspace}).
	Let $v\in L^1_w$, $\Norm{v}_{L^1_w}\leq 1$, $h>0$ (the case $h<0$ is analogous) and $T_{\xi,\: h}[v](\cdot) := T_{\xi}[v](\cdot + h)$,
	\begin{align*}
		&\Norm{T_{\xi, \: h}[v] - T_{\xi}[v]}_{L^1_w} = \int\limits_0^{z_0} \rec{b(z)} \left\vert \int\limits_0^{z+h}\int\limits_x^{z_0} \frac{\beta_m(z')k(x,z')}{b(z')} v(z') \: dz'  e^{-\int_x^{z+h} \alpha(y)\: dy} \: dx \right.\\
		&\:\mathrel{\phantom{\leq}} \left. - \int\limits_0^{z}\int\limits_x^{z_0} \frac{\beta_m(z')k(x,z')}{b(z')} v(z') \: dz' e^{-\int_x^{z} \alpha(y)\: dy} \: dx \right\vert \: dz\\
		&\:\leq \int\limits_0^{z_0} \rec{b(z)} \left(\int\limits_0^z \int\limits_x^{z_0} \frac{\beta_m(z')k(x,z')}{b(z')} \abs{v(z')} \: dz' \abs{e^{-\int_x^{z+h} \alpha(y)\: dy}-e^{-\int_x^{z} \alpha(y)\: dy}}\: dx \right.\\
		&\:\mathrel{\phantom{\leq}}\left.+ \int\limits_z^{z+h} \int\limits_x^{z_0} \frac{\beta_m(z')k(x,z')}{b(z')} \abs{v(z')} \: dz' \abs{e^{-\int_x^{z+h} \alpha(y)\: dy}}\: dx\right) dz\\
		&\:\leq \overline{\beta}_m \Norm{k}_{\infty} \left( \int\limits_0^{z_0}\int\limits_0^z \rec{b(z)} e^{-\int_x^{z} \alpha_r(y)\: dy} \abs{e^{-\int_z^{z+h} \alpha(y)\: dy}-1}\: dxdz\right. \\
		&\:\mathrel{\phantom{\leq}} \left.+ \int\limits_0^{z_0} \int\limits_{z}^{z+h} \rec{b(z)} e^{-\int_x^{z+h} \alpha_r(y)\: dy}\: dxdz \right)\:\xrightarrow[]{h\rightarrow 0^+} 0,
	\end{align*}
	uniformly in $v\in L^1_w$ with $\Norm{v}_{L^1_w}\leq 1$. Assumption \ref{Ass-ad} ensures that the first summand is bounded for every $h>0$. Therefore, $T_{\xi}$ is a compact operator. 
\end{proof}

\medskip
\begin{lemma}\label{Lem-Tsrad}
	Let $\xi\in\IR$ with $\xi>-(\underline{\beta}_m+\underline{\mu})$, then $r(T_{\xi})\in \sigma(T_\xi)$ and if $r(T_{\xi})>0$ then it is an eigenvalue of $T_{\xi}$ with a non-negative eigenfunction. 
\end{lemma}

\medskip
\begin{proof}
	For the first part of the Lemma we use a Theorem by Bonsall \cite[Theorem 1]{Bonsall1955}. 
	$L^1_w$ is a partially ordered Banach space with the relation $f\leq g$ iff $f(z) \leq g(z)$ for a.e. $z\in(0,z_0)$ and $L^{1\: +}_w:=\{ f\in L^1_w: \: f\geq 0 \text{ a.e.}\}$ is a normal cone as it is a non-empty, closed set and $\Norm{f+g}_{L^1_w}\geq \Norm{f}_{L^1_w}$ for all $f$, $g\in L^{1\: +}_w$. It holds that $L^1_w = L^{1\: +}_w - L^{1\: +}_w$. $T_{\xi}$ is a bounded linear operator and for $v\in L^1_w$ with $v\geq 0$ a.e.,
	\begin{align*}
	T_{\xi}[v](z) &= \int\limits_0^z \int\limits_x^{z_0} \frac{\beta_m(z')k(x,z')}{b(z')} v(z') \: dz' \: e^{-\int_{x}^{z} \alpha(y) \: dy}\: dx \geq 0 \text{ for all }z\in(0,z_0).
	\end{align*}
	Therefore, $T_{\xi}: L^1_w \rightarrow L^1_w$ is a endormorphism (in the sense of \cite{Bonsall1955}) and $r(T_{\xi})\in\sigma(T_{\xi})$ by \cite[Theorem 1]{Bonsall1955}. 
	
	By Lemma \ref{Lem-Tcomp}, $T_{\xi}$ is a compact operator. Since $T_{\xi}$ is a positive operator and $L^1_w = L^{1\:+}_w - L^{1\:+}_w$ the second part of the Lemma follows directly from Krein-Rutman Theorem (see e.g. \cite[Theorem 5.1]{Heijmans1986}).
\end{proof} 

\bigskip
\noindent We make an additional assumption on the plasmid segregation kernel $k$:
\begin{enumerate}[label=(A\arabic*), leftmargin=*]
	\setcounter{enumi}{5}
	\item $k(z,z')>0$ for all $z\in(0,z')$, $z'\in(m,z_0)$.\label{Ass-k2}
\end{enumerate}
We will use this assumption to show that the spectral radius $r(T_{\xi})$ is a simple eigenvalue of $T_{\xi}$ and thereby we obtain that the dominant real eigenvalue is simple.

\medskip
\begin{example}
	Assumption \ref{Ass-k2} holds e.g. if $k(z,z')=\frac{2}{z'} \Phi\left(\frac{z}{z'}\right) \chi_{\Omega}(z,z')$ and $\Phi(\xi) >0$ for all $\xi\in(0,1)$.
\end{example}

\medskip
Usually, one uses e.g. the Krein-Rutman Theorem in its strong form for strictly positive operators (see e.g. \cite[Theorem 1]{Dautray1990}) 
or the theory of non-support operators developed by Sawashima \cite{Sawashima1964} to show that the spectral radius is a simple eigenvalue \cite{Heijmans1986}. However, the operator $T_{\xi}$ is neither strictly positive nor a non-support operator as it maps every function with essential support in $[0,m]$ to 0. Therefore, in the following lemma we consider the operator $T_{\xi}$ on the interval $[m,z_0]$ and then use the strong Krein-Rutman Theorem.

\medskip
\begin{lemma}\label{Lem-Uniqueness}
	Let $\xi\in\IR$ with $\xi>-(\underline{\beta}_m+\underline{\mu})$. If $r(T_{\xi})>0$ then $r(T_{\xi})$ is a simple eigenvalue of $T_{\xi}$, i.e. there is a unique non-negative eigenfunction for $r(T_{\xi})$. 
\end{lemma}

\medskip
\begin{proof}
	We restrict the operator $T_{\xi}$ to the interval $[m,z_0]$, i.e. we consider the operator $\tilde{T}_{\xi}[v](z) := \chi_{[m,z_0]}(z) T_{\xi}[v](z)$ which maps $\tilde{L}^1_w:= L^1\left((m,z_0),\frac{dz}{b(z)}\right)$ to $\tilde{L}^1_w$.
	
	\indent \textit{Step 1: $r(\tilde{T}_{\xi})$ is a simple eigenvalue of $\tilde{T}_{\xi}$}\\
	The operator $\tilde{T}_{\xi}$ is compact if $\Re(\xi)>-(\underline{\beta}_m + \underline{\mu})$ (this can be shown as compactness of $T_{\xi}$ in Lemma \ref{Lem-Tcomp}). Let $\xi\in\IR$ and $v\in \tilde{L}^1_w$, $v\geq 0$, $v\neq 0$, then by assumption \ref{Ass-k2} for all $x<m$
	\begin{align*}
	\int\limits_x^{z_0} \frac{\beta_m(z') k(x,z')}{b(z')} v(z') \: dz' > 0.
	\end{align*}
	Therefore, for all $z\in(m,z_0)$
	\begin{align*}
	\tilde{T}_{\xi}[v](z) = \chi_{[m,z_0]}(z)\int\limits_0^z \int\limits_x^{z_0} \frac{\beta_m(z') k(x,z')}{b(z')} v(z') \: dz'\: e^{-\int_x^z \alpha(y)\: dy}\: dx > 0
	\end{align*}
	and $\tilde{T}_{\xi}$ is a strictly positive operator for $\xi\in\IR$. By the Krein-Rutman Theorem (see e.g. \cite{Heijmans1986}) $\tilde{r}:= r(\tilde{T}_{\xi})$ is a simple eigenvalue of $\tilde{T}_{\xi}$. 
	
	\textit{Step 2: $\lambda\in\IC\setminus\{0\}$ is an eigenvalue of $\tilde{T}_{\xi}$ iff $\lambda$ is an eigenvalue of $T_{\xi}$}\\
	Let $\lambda$ be an eigenvalue of $T_{\xi}$ with eigenfunction $v$, then $\lambda$ is an eigenvalue of $\tilde{T}_{\xi}$ with eigenfunction $\tilde{v}:= \left. v\right\vert_{[m,z_0]}$.\\
	Let $\lambda$ be an eigenvalue of $\tilde{T}_{\xi}$ with eigenfunction $\tilde{v}$. Define $v:= \rec{\lambda} T_{\xi}[\tilde{v}]$, then $v\in L^1_w$, $\tilde{v}= \left.v\right\vert_{[m,z_0]}$, and
	\begin{align*}
	T_{\xi}[v] = T_{\xi}[\left.v\right\vert_{[m,z_0]}] = T_{\xi}[\tilde{v}]=\lambda v.
	\end{align*}
	Therefore, $v$ is an eigenfunction of $T_{\xi}$ with eigenvalue $\lambda$.
	
	\textit{Step 3: A simple non-zero eigenvalue of $\tilde{T}_{\xi}$ is also a simple eigenvalue of $T_{\xi}$}\\
	Let $\lambda\in\IC\setminus\{0\}$ be a simple eigenvalue of $\tilde{T}_{\xi}$ with unique eigenfunction. From the previous step we already know that $\lambda$ is also an eigenvalue of $T_{\xi}$. Assume there are two different eigenfunctions $v_1$, $v_2\in L^1_w$ for the eigenvalue $\lambda$. Then $\left. v_1 \right\vert_{[m,z_0]}$ and $\left. v_2 \right\vert_{[m,z_0]}$ are eigenfunctions of $\tilde{T}_{\xi}$ for $\lambda$. Therefore, as $\lambda$ is a simple eigenvalue $\left. v_1 \right\vert_{[m,z_0]}= c \left. v_2 \right\vert_{[m,z_0]}$ for some $c\in\IC\setminus\{0\}$. W.l.o.g. $c=1$. 
	Hence,
	\begin{align*}
	\lambda v_1 = T_{\xi}[v_1] = T_{\xi}[\left. v_1 \right\vert_{[m,z_0]}] = T_{\xi}[\left. v_2 \right\vert_{[m,z_0]}] = T_{\xi}[v_2] = \lambda v_2.
	\end{align*}
	Therefore, $v_1=v_2$ which is a contradiction to $v_1$ and $v_2$ being different eigenfunctions.
	
	\textit{Step 4: Conclusion}\\
	By step 2 and compactness of both $T_{\xi}$ and $\tilde{T}_{\xi}$ we know that $\sigma(T_{\xi})\setminus\{0\} = \sigma(\tilde{T}_{\xi})\setminus\{0\}$. Therefore $\tilde{r} = r(T_{\xi})$ and by steps 1 and 3 it follows that $r(T_{\xi})$ is a simple eigenvalue of $T_{\xi}$. The non-negativity of the eigenfunction follows from Lemma \ref{Lem-Tsrad}.
\end{proof}

\medskip
\begin{lemma}\label{Lem-realsp}
	The largest real eigenvalue $\lambda_d$ of $A$ is a simple eigenvalue and an isolated point of $\sigma(A)\cap\IR$. 
\end{lemma}

\medskip
\begin{proof}
	For $\xi\in\IR$ the map $\xi\mapsto r(T_{\xi})$ is continuous as $T_{\xi}$ is compact for $\xi>-(\underline{\beta}_m+\underline{\mu})$ (see e.g. \cite[Theorem 2.1]{Degla2008}) and strictly decreasing as $T_{\xi}$ is strictly decreasing in $\xi$. 
	
	For $\xi > 2\overline{\beta}_m-\underline{\beta}_m-\underline{\mu}$, $r(T_{\xi})\leq \Norm{T_{\xi}}_{L^1_w}<1$ and by Theorem \ref{ThmExistence} there is some $\xi^{\ast}\in [-\underline{\mu},2\overline{\beta}_m-\underline{\beta}_m-\underline{\mu}]$ such that $\xi^{\ast}\in\sigma(A)$ thus $1\in\sigma(T_{\xi^{\ast}})$ and $r(T_{\xi^{\ast}})\geq 1$. Since $\xi\mapsto r(T_{\xi})$ is continuous and decreasing there is a unique $\hat{\xi}\in[-\underline{\mu},2\overline{\beta}_m-\underline{\beta}_m-\underline{\mu}]$ such that $r(T_{\hat{\xi}})=1$, i.e. $\hat{\xi}\in \sigma(A)$.
	
	For $\xi>\hat{\xi}$, $r(T_{\xi})<1$, i.e. $1\in\rho(T_{\xi})$ and hence $\xi\in\rho(A)$. As $r(T_{\hat{\xi}})=1$ is an eigenvalue $>0$ of a compact operator, it is an isolated point in the spectrum, i.e. there is $\varepsilon>0$ such that $[1-\varepsilon,1+\varepsilon] \cap \sigma(T_{\hat{\xi}}) = \{ 1 \}$. By continuity and monotonicity of the map $\xi\mapsto r(T_{\xi})$ there is therefore a $\varepsilon>0$ such that for $\xi\in[\hat{\xi}-\varepsilon,\hat{\xi})$, $1\notin \sigma(T_{\xi})$, i.e. $\xi\notin\sigma(A)$ and $\hat{\xi}$ is an isolated point in the real spectrum of $A$.
	
	It remains to be shown that $\hat{\xi}$ is a simple eigenvalue of $A$. Denote the unique (up to a constant) eigenfunction corresponding to $1=r(T_{\hat{\xi}})$ by $\hat{v}\in L^1_w$ ($\hat{v}$ is unique by Lemma \ref{Lem-Uniqueness}). Then, $\hat{v}=T_{\hat{\xi}}[\hat{v}]$ and therefore $\hat{v}$ is the unique solution of
	\begin{align*}
	\hat{v}'(z) + \frac{\hat{\xi} + \beta_m(z) + \mu(z)}{b(z)} \hat{v}(z) - \int\limits_z^{z_0} \frac{\beta_m(z')k(z,z')}{b(z')}\hat{v}(z')\: dz' =0.
	\end{align*}
	Define for $z\in(0,z_0)$, $\hat{u}(z):=\frac{\hat{v}(z)}{b(z)}$, then $\hat{u}\in L^1((0,z_0))$ and $(\hat{\xi}I-A)\hat{u}=0$, i.e. $\hat{u}$ is the unique eigenfunction for the simple eigenvalue $\hat{\xi}$ of $A$ and $\lambda_d:=\hat{\xi}$. 
\end{proof}

\medskip
\begin{lemma}\label{Lem-CompSp} 
	If $\lambda \in \sigma(A)$, $\lambda\neq \lambda_d$, then $\Re(\lambda)<\lambda_d$.
\end{lemma}

\medskip
\begin{proof}
	The proof follows the proof of Theorem 6.13 in \cite{Heijmans1986}. 
	
	Let $\lambda\in \sigma(A)$ with $\Re(\lambda)>-(\underline{\beta}_m + \underline{\mu})$ and denote with $\lambda_d$ the largest real eigenvalue of $A$.
	As $\lambda\in\sigma(A)$, $1 \in \sigma(T_{\lambda})$ and therefore, by compactness of $T_{\lambda}$, 1 is an eigenvalue. Denote the corresponding eigenfunction by $f\in L^1_w\setminus\{0\}$. Then,
	\begin{align*}
	\abs{f} = \abs{T_{\lambda}[f]}&=\abs{\int\limits_0^z \int\limits_x^{z_0} \frac{\beta_m(z')k(x,z')}{b(z')} f(z') \: dz' \: e^{-\int_{x}^{z} \alpha(y) \: dy}\: dx}\\
	&\leq \int\limits_0^z \int\limits_x^{z_0} \frac{\beta_m(z')k(x,z')}{b(z')} \abs{f}(z') \: dz' \: e^{-\int_{x}^{z} \Re(\alpha(y)) \: dy}\: dx \\
	&= T_{\Re(\lambda)}[\abs{f}].
	\end{align*}
	As $r(T_{\Re(\lambda)})\neq 0$ is an eigenvalue of $T_{\Re(\lambda)}$ and $T_{\Re(\lambda)}$ is compact, $r(T_{\Re(\lambda)})$ is also a eigenvalue of the dual operator $T^{\ast}_{\Re(\lambda)}$ (see e.g. \cite[Theorem 2]{Yosida1995}).
	We denote the eigenfunction of $T^{\ast}_{\Re(\lambda)}$ corresponding to the eigenvalue $r(T_{\Re(\lambda)})$ by $f^{\ast}\in L^{\infty}_w := \left\lbrace f\in L^{\infty}((0,z_0)):\: \frac{f(z)}{b(z)}\in L^{\infty}((0,z_0)) \right\rbrace$. 
	Let $f^{\ast}$ be normalized s.t. $\left\langle f^{\ast},f\right\rangle=1$.
	Taking duality pairings in $\abs{f} \leq T_{\Re(\xi)}[\abs{f}]$ yields
	\begin{align*}
	r(T_{\Re(\lambda)}) \left\langle f^{\ast},\abs{f} \right\rangle \geq \left\langle f^{\ast},\abs{f} \right\rangle,
	\end{align*}
	therefore $r(T_{\Re(\lambda)}) \geq 1$. As $\lambda \mapsto r(T_{\lambda})$ is decreasing for $\lambda \in \IR$ and $r(T_{\lambda_d})=1$ (see the proof of Lemma \ref{Lem-realsp}) this implies that $\Re(\lambda)\leq \lambda_d$.
	
	Suppose that $\lambda = \lambda_d + i\eta$, then as $\Re(\lambda)=\lambda_d$ we obtain $T_{\lambda_d}[\abs{f}] \geq \abs{f}$.
	Assume that $T_{\lambda_d}[\abs{f}] > \abs{f}$, then taking duality pairings again, $\left\langle f,f^{\ast} \right\rangle > \left\langle f,f^{\ast} \right\rangle$ which is a contradiction. Therefore, $T_{\lambda_d}[\abs{f}] = \abs{f}$.
	
	Let $f_d$ be the non-negative eigenfunction of $T_{\lambda_d}$ for the simple eigenvalue 1 (by Lemma \ref{Lem-realsp}). Therefore, $\abs{f}=c f_d$ for some constant $c\in\IR$, 		i.e. $f(z) = c f_d(z) e^{i g(z)}$ for some real valued function $g(z)$. W.l.o.g. $c=1$. 
	Substituting this in $T_{\lambda_d}[f_d]=\abs{T_{\lambda}[f]}$ yields
	\begin{align*}
	&\int\limits_0^z \int\limits_x^{z_0} \frac{\beta_m(z') k(x,z')}{b(z')} f_d(z') \: dz' e^{-\int_x^z \frac{\lambda_d + \beta_m(y) + \mu(y)}{b(y)}\: dy}\: dx =\\
	&\quad = \abs{\int\limits_0^z \int\limits_x^{z_0} \frac{\beta_m(z') k(x,z')}{b(z')} f_d(z') e^{ig(z')} \: dz' e^{-\int_x^z \frac{\lambda_d + i\eta + \beta_m(y) + \mu(y)}{b(y)}\: dy}\: dx}.
	\end{align*}
	By \cite[Theorem 1.39]{Rudin1886}, there exists $a\in \IC$ with $\abs{a}=1$ such that
	\begin{align*}
	f_d(z) = \abs{f_d(z)e^{i\left(g(z')-\eta \int_x^z\rec{b(y)}\:dy \right)}} = a f_d(z)e^{i\left(g(z')-\eta \int_x^z\rec{b(y)}\:dy \right)}.
	\end{align*}
	Let $\varphi\in \IR$ such that $a=e^{i\varphi}$, then
	\begin{align}
	g(z') - \eta \int_x^z\rec{b(y)}\:dy + \varphi = 2\pi\cdot k, \quad \text{ for }k\in\IZ.\label{Eq-g-eta-phi}
	\end{align}
	Therefore, the left-hand side is independent of $x$ and $z$ and $\eta =0$. Thus, we have now shown that if $\lambda=\lambda_d + i\eta \in \sigma(A)$ then $\eta=0$, i.e. $\lambda=\lambda_d$. Overall, if $\lambda\in\sigma(A)$ and $\lambda\neq \lambda_d$ then $\Re(\lambda)<\lambda_d$.
\end{proof}
	
\section{\texorpdfstring{Construction of $\mathcal{U}$}{Construction of U}}
\label{S-Numerics}

For the construction of $\mathcal{U}$ we consider the case that $\beta$ and $\mu$ are constant with $\beta>0$ and $\mu\geq 0$, $b(z)=\frac{b_0}{z_0}z(z_0-z)$ for some $b_0>0$, and $k$ is scalable, i.e. $k(z,z')=\frac{2}{z'} \Phi\left(\frac{z}{z'}\right) \chi_{\Omega}(z,z')$ (cf. Example \ref{Ex-Ass}(a)).

For constant $\beta$ and $\mu$ it is known that $\lambda=\beta-\mu$. Furthermore, for $\Phi\equiv 1$ we can compute the solution $\mathcal{U}$ for \eqref{EigProb} explicitly. Note that if $\Phi\equiv 1$ then $k$ does not satisfy Assumption \ref{Ass-ad}. Nevertheless, we consider this case as we are then able to compare the numerical simulation for $\mathcal{U}$ with the exact solution and thereby assess the simulation.

\medskip
\begin{example}\label{Ex-sol}
For $\beta>0$ and $\mu\geq 0$ constant, $\lambda=\beta-\mu$, $b(z)=\frac{b_0}{z_0}z(z_0-z)$, and $k$ scalable with $\Phi\equiv 1$ the function $\mathcal{U}(z) = Cz^{-\alpha}(z_0-z)^{\alpha}$ with $C>0$ and $\alpha := \frac{\lambda + \beta + \mu}{b_0}=\frac{2\beta}{b_0}$ is a solution for the eigenproblem \eqref{EigProb} on $[m,z_0]$.
\end{example}

\medskip
We use again the transformation $v(z) = b(z)\mathcal{U}(z)$ (c.f. Section \ref{S-Spectrum}), then $\mathcal{U}$ is a solution for the eigenproblem \eqref{EigProb} on $[m,z_0]$ iff $v$ is a solution for
\begin{equation}\label{Eq-v2}
	\begin{aligned}
		&v'(z) + \alpha z_0 \frac{v(z)}{z(z_0-z)} = \alpha z_0 \int\limits_z^{z_0} \frac{\Phi\left(\frac{z}{z'}\right)\:v(z')}{(z')^2(z_0-z')}\: dz',
	\end{aligned}
\end{equation}
where $\alpha =\frac{2\beta}{b_0}$. By Example \ref{Ex-sol}, for $\Phi\equiv 1$ the function $v(z)= C z^{1-\alpha}(z_0-z)^{\alpha}$ is a solution of \eqref{Eq-v2}. 
Therefore, we use the ansatz $v(z)= (z_0 -z)^{\alpha} g(z)$ with $g(z_0)=1$.

\medskip
\begin{lemma}\label{Lem-gv}
	If there is a solution $g$ with $g\in\cont{0}([m,z_0])\cap\cont{1}((m,z_0))$ for 
	\begin{equation}\label{Eq-g}
		\begin{aligned}
			&g'(z) + \frac{\alpha}{z} g(z) = \frac{\alpha_0 z_0}{(z_0-z)^{\alpha}} \int\limits_{z}^{z_0} \frac{\Phi\left( \frac{z}{z'}\right)}{(z')^2} (z_0-z')^{\alpha -1}g(z')\: dz', \quad g(z_0)=1 
		\end{aligned}
	\end{equation}
	then $v$ with $v(z):= C (z_0-z)^{\alpha} g(z)$ for some $C>0$ is a solution for \eqref{Eq-v2} and $v\in\cont{1}((m,z_0))$.
\end{lemma}

\medskip
\begin{remark}
	In Lemma \ref{Lem-gv} we do not have equivalence as there can only be a function $g$ with $v(z)=C(z_0-z)^{\alpha}g(z)$ and $g(z_0)=1$ if $v(z)$ behaves like $(z_0-z)^{\alpha}$ in a neighborhood of $z_0$. 
	Otherwise, either $g(z_0)=0$ or $\lim\limits_{z\rightarrow z_0^-} g(z) =\infty$.
	By Example \ref{Ex-sol} we know that at least for $\Phi\equiv 1$, $v(z)\sim (z_0-z)^{\alpha}$ near $z_0$.
\end{remark}

\medskip
\begin{proof}
	Define $v(z):= (z_0-z)^{\alpha} g(z)$. As $g$ is a solution for \eqref{Eq-g},
	\begin{align*}
	v'(z) &= g'(z) (z_0-z)^{\alpha} + g(z) \alpha (z_0-z)^{\alpha-1} (-1)\\
	&= (z_0-z)^{\alpha} \left( -\frac{\alpha}{z} g(z) + \frac{\alpha z_0}{(z_0-z)^{\alpha}} \int\limits_{z}^{z_0}\Phi\left( \frac{z}{z'}\right) g(z') (z')^{-2}(z_0-z')^{\alpha -1}\: dz' \right)\\
	&\phantom{{}={}} -\frac{\alpha}{z_0-z} v(z)\\
	&= -\frac{\alpha z_0}{z(z_0-z)} v(z) + \alpha z_0 \int\limits_{z}^{z_0} \Phi\left(\frac{z}{z'}\right) \frac{v(z')}{(z')^2 (z_0-z')}\: dz'.
	\end{align*}
\end{proof}

\medskip
\begin{lemma}
	There exists a unique solution $g \in \cont{0}([m,z_0])\cap\cont{1}((m,z_0))$ for \eqref{Eq-g}. \label{Lem-g}
\end{lemma}

\medskip
\begin{proof}
	Variation of parameters yields the equation
	\begin{align*}
	g(x) &= \frac{\alpha z_0}{x^{\alpha}} \int\limits_{z_0}^{x} \left(\frac{z}{z_0-z}\right)^{\alpha} \int\limits_{z}^{z_0} \Phi\left( \frac{z}{z'} \right) (z')^{-2} (z_0-z')^{\alpha - 1} g(z') \: dz'dz + \left(\frac{z_0}{x}\right)^{\alpha},
	\end{align*}
	for $x\in (0,z_0]$. Let $I_1 := (z_0/x)^{\alpha}$.
	
	\noindent Let $a\in[m,z_0)$. Consider the operator $G: \cont{0}([a,z_0]) \rightarrow \cont{0}([a,z_0])$ defined by
	\begin{align*}
	G[g](x) := \frac{\alpha z_0}{x^{\alpha}} \int\limits_{z_0}^{x} \left(\frac{z}{z_0-z}\right)^{\alpha} \int\limits_{z}^{z_0} \Phi\left( \frac{z}{z'} \right) (z')^{-2} (z_0-z')^{\alpha - 1} g(z') \: dz'dz.
	\end{align*}
	Note that $G[g](x)$ for $x\in[a,z_0]$ only depends on $\left.g\right\vert_{[a,z_0]}$. If $\alpha\neq 1$, we estimate the supremum norm of $G[g]$ for some $g\in\cont{0}([a,z_0])$ using $\rec{(z')^2}\leq \rec{z^2}$ for $z\leq z'$,
	\begin{align*}
		\Norm{G[g]}_{\infty} &\leq \alpha z_0 \Norm{\Phi}_{\infty} \Norm{g}_{\infty} \\
		&\phantom{{}\leq{}}\sup\limits_{x\in[a,z_0]} x^{-\alpha} \abs{\int\limits_{z_0}^{x}\left( \frac{z}{z_0-z} \right)^{\alpha} \frac{1}{z^2} \int\limits_{z}^{z_0} (z_0-z')^{\alpha-1}\:dz' dz }\\
		&= z_0 \Norm{\Phi}_{\infty} \sup\limits_{x\in[a,z_0]} x^{-\alpha} \abs{\int\limits_{z_0}^{x}z^{\alpha-2}\:dz } \Norm{g}_{\infty}\\
		&= \frac{z_0}{\abs{\alpha-1}} \Norm{\Phi}_{\infty} \sup\limits_{x\in[a,z_0]} x^{-\alpha} \abs{ x^{\alpha-1}-z_0^{\alpha-1} } \Norm{g}_{\infty}\\
		&= \frac{1}{\abs{\alpha-1}} \Norm{\Phi}_{\infty} \abs{ \frac{z_0}{a} - \left( \frac{z_0}{a} \right)^{\alpha} } \Norm{g}_{\infty}.
	\end{align*}
	Therefore, $G: \cont{0}([a,z_0]) \rightarrow \cont{0}([a,z_0])$ is a bounded operator.
	It holds that
	\begin{align*}
	\lim\limits_{a\rightarrow z_0} \abs{ \frac{z_0}{a} - \left( \frac{z_0}{a} \right)^{\alpha} } = 0.
	\end{align*}
	Therefore, there exists some $\varepsilon\in (0,z_0)$ such that $G$ is a contraction on $[z_0-\varepsilon,z_0]$ and $I_1$ is bounded on $[z_0-\varepsilon,z_0]$. The case $\alpha =1$ is analogous. 
	Hence, by the Banach Fixed Point Theorem, there exists an $\varepsilon\in(0,z_0)$ such that there is a unique continuous solution $g: [z_0-\varepsilon,z_0]\rightarrow \mathbb{R}$ of \eqref{Eq-g}.
	
	We construct a continuously differentiable solution iteratively: assume we are given a unique continuous solution $g:[a,z_0] \rightarrow \IR$ on $[a,z_0]$ for some $a \in (0,z_0)$. By variation of parameters, for $x\in (0,a]$, 
	\begin{align*}
	g(x) &= \frac{\alpha z_0}{x^{\alpha}} \int\limits_{a}^{x} \left(\frac{z}{z_0-z}\right)^{\alpha} \int\limits_{z}^{a} \Phi\left( \frac{z}{z'} \right) (z')^{-2} (z_0-z')^{\alpha - 1} g(z') \: dz'dz\\
	&\phantom{=} +\frac{\alpha z_0}{x^{\alpha}} \int\limits_{a}^{x} \left(\frac{z}{z_0-z}\right)^{\alpha} \int\limits_{a}^{z_0} \Phi\left( \frac{z}{z'} \right) (z')^{-2} (z_0-z')^{\alpha - 1} g(z') \: dz'dz + \left( \frac{a}{x}\right)^{\alpha} g(a).
	\end{align*}
	Denote the second summand by $I_2(x)$ and the third summand by $I_3(x)$.
	We consider for some $\delta \in (0,a)$ the operator $G_a: \cont{1}([a-\delta, a]) \rightarrow \cont{1}([a-\delta, a])$ given by
	\begin{align*}
	G_a[g](x) := \frac{\alpha z_0}{x^{\alpha}} \int\limits_{a}^{x} \left(\frac{z}{z_0-z}\right)^{\alpha} \int\limits_{z}^{a} \Phi\left( \frac{z}{z'} \right) (z')^{-2} (z_0-z')^{\alpha - 1} g(z') \: dz'dz.
	\end{align*}
	Analogously to before, for $x\in [a-\delta,a]$ and $\alpha\neq 1$ we estimate the supremum norm
	\begin{align*}
	\Norm{G_a[g]}_{\infty} &\leq \alpha z_0 \Norm{\Phi}_{\infty} \Norm{\left. g\right\vert_{[a-\delta,a]}}_{\infty} \\
	&\mathrel{\phantom{\leq}} \cdot \sup\limits_{x\in[a-\delta,a]} x^{-\alpha} \abs{\int\limits_{a}^{x}\left( \frac{z}{z_0-z} \right)^{\alpha} \frac{1}{z^2} \int\limits_{z}^{a} (z_0-z')^{\alpha-1}\:dz' dz}\\
	&\leq z_0 \Norm{\Phi}_{\infty} \sup\limits_{x\in[a-\delta,a]} x^{-\alpha} \abs{\int\limits_{x}^{a} z^{\alpha-2} \: dz} \Norm{\left. g\right\vert_{[a-\delta,a]}}_{\infty}\\
	&\leq \frac{z_0}{\abs{\alpha-1} a} \Norm{\Phi}_{\infty} \sup\limits_{x\in[a-\delta,a]} \abs{ \left(\frac{a}{x}\right)^{\alpha}-\frac{a}{x} } \Norm{\left. g\right\vert_{[a-\delta,a]}}_{\infty}\\
	&\leq \frac{z_0}{\abs{\alpha-1} a} \Norm{\Phi}_{\infty} \sup\limits_{x\in[a-\delta,a]} \abs{ \left(\frac{a}{x}\right)^{\alpha}-\frac{a}{x} } \Norm{\left. g\right\vert_{[a-\delta,a]}}_{\cont{1}([a-\delta,a])}.
	\end{align*}
	Furthermore,
	\begin{align*}
	\frac{d}{dx} G_a[g] (x) &=\left(-\frac{\alpha}{x}\right) G_a[g](x) + \frac{\alpha z_0}{x^{\alpha}} \left( \frac{x}{z_0-x} \right)^{\alpha} \int\limits_{x}^{a} \frac{\Phi\left(\frac{x}{z'}\right)}{(z')^2} (z_0-z')^{\alpha-1} g(z')\: dz'\\
	&=\left(-\frac{\alpha}{x}\right) G_a[g](x) + \frac{\alpha z_0}{(z_0-x)^{\alpha}} \int\limits_{x}^{a} \Phi\left(\frac{x}{z'}\right) (z')^{-2} (z_0-z')^{\alpha-1} g(z')\: dz',
	\end{align*}
	and for $x\in[a-\delta,a]$,
	\begin{align*}
	\Norm{\frac{d}{dx} G_a[g]}_{\infty} & \leq \frac{\alpha}{a-\delta} \Norm{G_a[g]}_{\infty} + \alpha z_0 \Norm{\Phi}_{\infty} \Norm{\left. g\right\vert_{[a-\delta,a]}}_{\infty}\\
	&\mathrel{\phantom{\leq}} \max\limits_{x\in[a-\delta,a]} \frac{1}{(z_0-x)^{\alpha}} \int_{x}^{a} (z')^{-2} (z_0-z')^{\alpha-1} \: dz'.
	\end{align*}
	Because of
	\begin{align*}
	&\max\limits_{x\in[a-\delta,a]} \frac{1}{(z_0-x)^{\alpha}} \int_{x}^{a} (z')^{-2} (z_0-z')^{\alpha-1} \: dz' \leq\\
	&\qquad\leq \frac{1}{(a-\delta)^2} \frac{1}{(z_0-a)^{\alpha}} \max\limits_{x\in[a-\delta,a]} \int\limits_{x}^{a} (z_0-z')^{\alpha-1}\:dz\\
	&\qquad = \frac{1}{(a-\delta)^2} \frac{1}{(z_0-a)^{\alpha}} \frac{1}{\alpha} \max\limits_{x\in[a-\delta,a]} \left( (z_0-x)^{\alpha}-(z_0-a)^{\alpha} \right)\\
	&\qquad = \frac{1}{(a-\delta)^2} \frac{1}{(z_0-a)^{\alpha}} \frac{1}{\alpha} \left( (z_0-(a-\delta))^{\alpha}-(z_0-a)^{\alpha} \right)\\
	&\qquad \rightarrow 0 \quad \text{for } \delta\rightarrow 0,
	\end{align*}
	and $a\in(0,z_0)$ it holds that
	\begin{align*}
	\Norm{G_a[g]}_{\cont{1}([a-\delta,a])}=\max\left\lbrace \Norm{G_a[g]}_{\infty},\:\Norm{\frac{d}{dx}G_a[g]}_{\infty} \right\rbrace < \Norm{g}_{\cont{1}([a-\delta,a])}
	\end{align*}
	for $\delta\in (0,a)$ sufficiently small, i.e. $G_a$ is a contraction on $\cont{1}([a-\delta,a])$. The inhomogeneities $I_2$ and $I_3$ are bounded on $[a-\delta,a]$, as $a-\delta >0$ and 
	\begin{align*}
	\abs{I_2} &\leq \frac{z_0}{\abs{\alpha - 1} a} \Norm{\Phi}_{L^{\infty}([0,1])} \Norm{\left.g\right\vert_{[a,z_0]}}_{\infty} \abs{ \frac{a}{a-\delta} -\left(\frac{a}{a-\delta}\right)^{\alpha}} <\infty.
	\end{align*}
	The case $\alpha=1$ is again analogous to the case $\alpha\neq 1$.
	By Banach's Fixed Point Theorem, there is a unique solution $g_1\in\cont{1}([a-\delta,a])$ with $g_1(a)=g(a)$ and hence there is a solution $g \in \cont{0}([a-\delta,z_0])\cap\cont{1}([a-\delta,a])$ of \eqref{Eq-g}.
	
	Proceeding iteratively, we can construct a unique solution $g\in \cont{0}([m,z_0])\cap\cont{1}((m,a])$ of \eqref{Eq-g} and since $a\in (0,z_0)$ arbitrary, $g\in\cont{0}([m,z_0])\cap\cont{1}((m,z_0))$. \end{proof}
	
\bigskip
The proof of Lemma \ref{Lem-g} yields a method to simulate the solution $g$ for \eqref{Eq-g}.
We simulated $g$ iteratively using the software \Rs{} \cite{R}: in each step, we simulated $g$ on the interval $[a-\delta,a]$ where $a$ is given from the previous step (in the first step $a=z_0$) and we chose $\delta>0$ such that the operator $G$ respectively $G_a$ in the proof of Lemma \ref{Lem-g} is a contraction on $\cont{0}([a-\delta,a])$, i.e.
\renewcommand*{\arraystretch}{1.5}
\begin{align*}
\delta= \left\lbrace \begin{array}{ll}
\argmin\limits_{\delta\in(0,a)} \:\abs{\frac{z_0}{\abs{\alpha-1}a} \Norm{\Phi}_{\infty} \abs{\frac{a}{a-\delta}-\left(\frac{a}{a-\delta}\right)^{\alpha}}-1}+10^{-5},\quad & \text{if }\alpha\neq 1, \\ 
\argmin\limits_{\delta\in(0,a)} \: \abs{z_0 \Norm{\Phi}_{\infty} \ln\left(\frac{a}{a-\delta}\right) \rec{a-\delta}-1}+10^{-5}, & \text{if }\alpha=1.
\end{array} \right.
\end{align*} 
For all points $x$, which were chosen equidistant between $a-\delta$ and $a$, we set $g_0(x):=g(a)$, where $g(a)$ is known from the previous step ($g(a)=g(z_0)=1$ for the first step). For $n\in\IN_0$, we update $g_n$ using
\begin{align*}
g_{n+1}(x) &:= \frac{\alpha_0 z_0}{x^{\alpha}} \int\limits_a^x \left(\frac{z}{z_0-z}\right)^{\alpha} \int\limits_z^a \frac{\Phi\left(\frac{z}{z'}\right)}{(z')^2} (z_0-z')^{\alpha-1} g_n(z')\:dz' dz + \left(\frac{a}{x}\right)^{\alpha}g_n(a)\\
&\mathrel{\phantom{=}} + \frac{\alpha_0 z_0}{x^{\alpha}} \int\limits_a^x \left(\frac{z}{z_0-z}\right)^{\alpha} \int\limits_a^{z_0} \Phi\left(\frac{z}{z'}\right) (z')^{-2}(z_0-z')^{\alpha-1} g_n(z')\:dz'dz,
\end{align*}
where the last summand was omitted for the first step with $a=z_0$. 
For the integration we used the \Rs{} function \texttt{integrate}.

The choice of $\delta$ guarantees that this iteration converges on $[a-\delta,a]$ as $G_a$ (defined as in the proof of Lemma \ref{Lem-g}, i.e. $G_a[g_n]$ is the first summand in the iteration) is a contraction. $g_n$ is updated until either $\max\limits_{x} \abs{g_n(x)-g_{n-1}(x)} <10^{-6}$ or $n=100$. Then, for $x\in [a-\delta,a]$, $g(x) := g_n(x)$.

We repeated the same procedure for the next step, with $a:=a-\delta$ until $a-\delta < m = 0.005$ (or maximally 1000 times).

From the solution $g$, we then computed $v(z):= (z_0-z)^{\alpha} g(z)$ and $\mathcal{U}(z)=\frac{v(z)}{b(z)}$ and normalized $\mathcal{U}$ such that $\int_{0.005}^{z_0} \mathcal{U}(z)\:dz =1$, where we numerically determined the integral using again the function \texttt{integrate}.

\begin{figure}
	\centering
	\includegraphics[width=\textwidth]{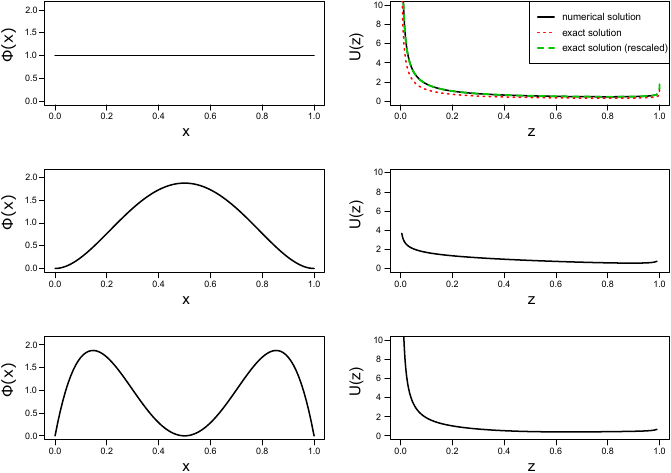}
	\caption{Numerical simulations of the eigenfunction $\mathcal{U}$ for $\beta = 0.4$, $\mu=0.1$, $b(z)=z(1-z)$, $\alpha=0.8$ and different $\Phi$ (the respective function $\Phi$ is shown in the box on the left-hand side). For $\Phi\equiv1$, the exact solution for $\mathcal{U}$ is plotted dotted for comparison, the exact solution with the same scaling as the numerical simulations is plotted dashed}
	\label{fig:1} 
\end{figure}

In Figure \ref{fig:1}, numerical simulations for $\mathcal{U}$ for different kernels, i.e. different $\Phi$, are plotted. In the case $\Phi\equiv 1$ we know the exact solution (c.f. Example \ref{Ex-sol}), we plotted the exact solution for comparison. If the exact solution is scaled in the same way as the numerical simulations, i.e. $\int_{0.005}^{z_0} \mathcal{U}(z)\:dz =1$, it agrees well with the numerical simulation (see Figure \ref{fig:1}, upper right figure).

From the numerical construction of the solutions $\mathcal{U}$ we observe that the distribution of plasmids for the uniform, unimodal, and bimodal segregation kernel look alike. In each case we find a pole at zero plasmids and only few bacteria with a high plasmid number. Therefore, simulations indicate that the segregation kernel does not influence the long-term plasmid distribution. 

\vspace{0.3cm}
\textbf{Acknowledgments} 
	\textit{I want to thank Johannes Müller for intensive discussions. 
	This work was funded by the German Research Foundation (DFG) priority program SPP1617 ``Phenotypic heterogeneity and sociobiology of bacterial populations'' (DFG MU 2339/2-2).\\
	This is a pre-print of an article published in Journal of Mathematical Biology. The final authenticated version is available online at: \url{https://doi.org/10.1007/s00285-018-1310-2}.}

\addcontentsline{toc}{chapter}{Bibliography}
\bibliographystyle{abbrv} 

\end{document}